\documentclass[12pt,reqno]{amsart}
\usepackage{mathrsfs}
\usepackage{amsfonts}
\usepackage{amsfonts}
\usepackage{amsfonts}
\usepackage{amsfonts}
\usepackage{amsfonts}
\usepackage{amsfonts}
\usepackage{amsfonts}
\usepackage{amsfonts}
\usepackage{amsfonts}
\usepackage{amsfonts}
\usepackage{amsfonts}
\usepackage{amsfonts}
\usepackage{amsfonts}
\usepackage{amsfonts}
\usepackage{amsfonts}
\usepackage{amsfonts}
\usepackage{amsfonts}
\usepackage{amsfonts}
\usepackage{a4wide}
\newtheorem{lemma}{Lemma}[section]

\newtheorem{theorem}{Theorem}[section]

\newtheorem{definition}{Definition}[section]

\let\Section=\section
\def\section{\setcounter{equation}{0}\Section}
\sf
\begin{document}
 \title[On a semilinear elliptic systems in Hyperbolic space]
 { On a semilinear elliptic systems in Hyperbolic space}

\author{Haiyang He$^*$}
\thanks{*Department of Mathematics, Hunan Normal University, Changsha Hunan 410081, The People's Republic of China
(hehy917@yahoo.com.cn)}
\maketitle

\vskip 0.2cm \arraycolsep1.5pt
\newtheorem{Lemma}{Lemma}[section]
\newtheorem{Theorem}{Theorem}[section]
\newtheorem{Definition}{Definition}[section]
\newtheorem{Proposition}{Proposition}[section]
\newtheorem{Remark}{Remark}[section]
\newtheorem{Corollary}{Corollary}[section]

\begin{abstract} In this paper, we consider systems of semilinear elliptic equations
\begin{equation}\label{eq:0.1}
\left\{
  \begin{array}{ll}
  \displaystyle
-\Delta_{\mathbb{H}^{N}}u=|v|^{p-1}v      ,    \\
\\
\displaystyle
-\Delta_{\mathbb{H}^{N}}v=|u|^{q-1}u      ,    \\
\end{array}
\right.
\end{equation}
in the whole of Hyperbolic space $\mathbb{H}^{N}$. We establish decay estimates and symmetry properties of  positive solutions.
Unlike the corresponding problem in Euclidean space $\mathbb{R}^N$,  we prove that there exists a nonnegative nontrivial solution of problem (\ref{eq:0.1}).

\end{abstract}
\vspace{3mm}
\noindent {\bf AMS} classification:  58J05,  35J60. \vspace{3mm}

\section {Introduction and main result}
In this article, we will study decay, symmetry and existence of solutions of
the following semilinear elliptic systems
\begin{equation}\label{eq:1.1}
\left\{
  \begin{array}{ll}
  \displaystyle
-\Delta_{\mathbb{H}^{N}}u=|v|^{p-1}v      ,    \\
\displaystyle
-\Delta_{\mathbb{H}^{N}}v=|u|^{q-1}u      ,    \\
\end{array}
\right.
 \end{equation}
 on Hyperbolic space $\mathbb{H}^{N}$, where $\Delta_{\mathbb{H}^{N}}$ denotes the Laplace-Betrami operator on $\mathbb{H}^{N}, N\geq 3,p$ and $q$ satisfy a suitable condition.

When posed in the Euclidean space $\mathbb{R}^N$, problem  (\ref{eq:1.1}) has two features. First, it is the Emden-Fowler equation
\begin{equation}\label{eq:1.2'}
-\Delta u=|u|^{p-1}u   \ in \ \mathbb{R}^N    .
 \end{equation}
Such a problem has been extensively studied, see for instance\cite{CGS}\cite{CL} \cite{GNN}\cite{GS1}\cite{GS2}  and references therein. Attention was focused
 on the existence and Liouville-type theorem
for solutions of problem (\ref{eq:1.2'}). There is a host of later important contributions to the subject, among them
we must mention the famous paper by \cite{GS2} where the Liouville-type theorem of problem (\ref{eq:1.2'}) was obtained. They
proved that the only non-negative solution of (\ref{eq:1.2'}) is $u=0$ when
\[1\leq p<\frac{N+2}{N-2}, \ \quad \ N\geq 3.\]
Second, it is the Hamiltonian type system
\begin{equation}\label{eq:1.3'}
\left\{
  \begin{array}{ll}
  \displaystyle
-\Delta u=|v|^{p-1}v,   \ in \ \mathbb{R}^N,\\
\displaystyle
-\Delta v=|u|^{q-1}u,  \ in \ \mathbb{R}^N.  \\
\end{array}
\right.
 \end{equation}
 Using a blow up technique, JieQing \cite{J} and M.A.Souto \cite{S}has been established the priori estimates for solutions
 of problem (\ref{eq:1.3'}). In \cite{FF}, Figueidedo and Felmer  proved that if $p>0, q>0$ are such that
 \[p, q\leq \frac{N+2}{N-2}, \ but \ not\ both \ are \ equal \ to \ \frac{N+2}{N-2},\]
 then the only non-negative solution of (\ref{eq:1.3'}) is the trivial one $u=0, v=0$. For more general nonlinear
 elliptic equations in the Euclidean space $\mathbb{R}^N$, we refer to \cite{BM}\cite{BS}\cite{M1}\cite{M2}\cite{SZ} and reference therein.

 It is also of interest to study problem (\ref{eq:1.2'}) and (\ref{eq:1.3'}) with respect to different ambient geometries
 in particular to see how curvature properties  affect the existence and nature of solutions. A recent paper by Mancini and Sandeep\cite{MS}
 have studied the existence / nonexistence and uniqueness of positive solution of the following elliptic equation
 \begin{equation}\label{eq:1.4'}
  \begin{array}{ll}
  \displaystyle
-\Delta_{\mathbb{H}^{N}}u=|u|^{p-1}u+\lambda u
\end{array}
 \end{equation}
 on Hyperbolic space $\mathbb{H}^{N}$. They proved that if $\lambda=0$ and $1<p<\frac{N+2}{N-2}$, then problem (\ref{eq:1.4'}) has a positive solution.
This result is contrasted with the result in Euclidean space due to \cite{GS2}.  Afterward, Bhakta
and Sandeep \cite{BS1} have investigated the priori estimates, existence of radial sign changing solutions of problem (\ref{eq:1.4'}).
In \cite{BGGV}, classification of radial solutions is done by Bonforte etc for problem (\ref{eq:1.4'}).

Our aim in this paper is to study the decay, symmetry  and existence of solution of problem (\ref{eq:1.1}). Our result should be
contrasted with a result due to Figueidedo and Felmer\cite{FF}. The difficulties in treating system (\ref{eq:1.1}) originate in at least
three facts. First, there is a lack of compactness due to the fact that we are working in $\mathbb{H}^{N}$ which is a noncompact manifold.
Second, due to the type of growth of the nonlinear term, we can not work with the usual space $H^1(\mathbb{H}^{N})$ and then we need inhomogeneous
Sobolev space. Third, although, we have a variational problem, the functional associated to it is strongly indefinite.

Now we are ready to state our main result. In section 2, we discuss the symmetry property of positive solution of problem (\ref{eq:1.1}).
\begin{theorem}\label{tm:1.1}
For $p,q$ satisfying
\begin{equation}\label{eq:1.2}
p,q \leq \frac{N+2}{N-2}
\end{equation}
Then, all positive solutions $(u,v)\in H^{1}(\mathbb{H}^{N})\times H^{1}(\mathbb{H}^{N}) $ of problem (\ref{eq:1.1}) are Hyperbolic symmetry, i.e. there is $x_{0}\in \mathbb{H}^{N}$, such that $(u,v)$ is constant on hyperbolic spheres centered at $x_{0}$.
\end{theorem}

In section 3, we prove a result on the decay of solutions of problem (\ref{eq:1.1}) as $|x|\rightarrow\infty$ and in fact, exponential decay. The results of
Section 3 is the following.
\begin{theorem}\label{tm:1.2}
For $p,q$ fulfilling (\ref{eq:1.2}), let $(u,v)\in H^{1}(\mathbb{H}^{N})\times H^{1}(\mathbb{H}^{N})$ be a positive radial solution of (\ref{eq:1.1}),
then $u'(r)<0, v'(r)<0$ for $r>0$ and
\[\lim\limits_{r\rightarrow \infty}u(r)=\lim\limits_{r\rightarrow \infty}v(r)=\lim\limits_{r\rightarrow \infty}u'(r)=\lim\limits_{r\rightarrow \infty}v'(r)=0.\]
Moreover, 
\[\lim\limits_{r\rightarrow \infty}\frac{\log u^{2}}{r}=\lim\limits_{r\rightarrow \infty}\frac{\log v^{2}}{r}=\lim\limits_{r\rightarrow \infty}\frac{\log u'^{2}}{r}=\lim\limits_{r\rightarrow \infty}\frac{\log v'^{2}}{r}=-2(N-1).\]
\end{theorem}

In section 4, we prove the existence of radial symmetric solutions of problem (\ref{eq:1.1}). Namely the following theorem.
\begin{theorem}\label{tm:1.3}
Suppose that $(p,q)$ satisfying
\[\frac{1}{p+1}+\frac{1}{q+1}>\frac{N-2}{N},\]
then problem (\ref{eq:1.1}) has at least one radial solution $(u,v)$.
\end{theorem}
Finally in section 5, we prove the existence of a ground state solution for the system (\ref{eq:1.1}).
\begin{theorem}\label{tm:1.4}
Suppose that $(p,q)$ satisfying (\ref{eq:1.2}),
then problem (\ref{eq:1.1}) has a nonnegative nontrivial ground state solutions.
\end{theorem}

\section  {Proof of theorem \ref{tm:1.1}}

The main purpose in this section is to prove Hyperbolic symmetry properties of solutions of (\ref{eq:1.1}). The way of proving this symmetry is by move planes, as originally introduced by Alex androff\cite{H}, later used by Serrin\cite{S2}, and extensively used in recent times after the world of Gidas-Ni-Nirenberg\cite{GNN}. The case on noncompact manifold was studied in \cite{ADG}\cite{AG}.

\textbf{Proof of Theorem \ref{tm:1.1}:}

Let $A_{t}$ be a one-parameter group of isometries of $\mathbb{H}^{N}$ which is $C^{1}(\mathbb{R}\times \mathbb{H}^{N},\mathbb{H}^{N})$ and $I$ be a reflection (i.e. $I$ is an isometry and $I^{2}$=Identity) satisfying the invariance condition
 \[A_{t}IA_{t}=I,\,\forall t\in \mathbb{R}.\]
 We translate the reflection $I$ using $A_{t}$ to define a one-parameter family of reflections
 \[I_{t}=A_{t}IA_{-t}.\]
 Let $U_{t}$ be the Hyper-surface of $\mathbb{H}^{N}$ which is fixed by $I_{t}$. We also  assume  that $\bigcup\limits_{t_{1}<t<t_{2}}U_{t}$  is  open for all $t_{1},\,t_{2}\in \mathbb{R}$, and $\bigcup\limits_{t\in \mathbb{R}}U_{t}=\mathbb{H}^{N}$. For $t\in \mathbb{R}$, define
 \[Q_{t}=\bigcup\limits_{-\infty<s<t}U_{s},\ \ \,Q^{t}=\bigcup\limits_{t<s<\infty}U_{s},\]
 then we have $I_{t}(Q_{t})\subset Q^{t}$,and $I_{t}(Q^{t})\subset Q_{t}$ for all $t\in \mathbb{R}$.
 For $t\in \mathbb{R}$ and $x\in Q_{t}$, we define $x_{t}=I_{t}(x)$, and
 \[u_{t}(x)=u(I_{t}(x)),\,v_{t}(x)=v(I_{t}(x)).\]

 Define
\[\Lambda=\{t\in \mathbb{R}:\forall \tau>t,\,u\geq u_{\tau} ,\,and \quad   v\geq v_{\tau},\ \ \, in \quad  Q_{\tau}\},\]
the first step of proof is to show that the set $\Lambda$ is non-empty, the second step is to prove that $\Lambda$ is bounded from below. Finally we will show that if $\bar{\Lambda}=\inf\Lambda$, then $u\equiv u_{\bar{\Lambda}},\,v\equiv v_{\bar{\Lambda}}$ in $Q_{\bar{\Lambda}}$.

Step 1. $\Lambda$ is non-empty.

Since $(u,v)\in H^{1}(\mathbb{H}^{N})\times H^{1}(\mathbb{H}^{N})$, we can take $(u_{\lambda}-u)^{+}$ or $(v_{\lambda}-v)^{+}$ as text function. Then we obtain
\[\begin{array}{ll}
&\ \  \int_{Q_{\lambda}}|\nabla_{\mathbb{H}^{N}}(u_{\lambda}-u)^{+}|^{2}\ dV_{\mathbb{H}^{N}}\\[2mm]
&= \int_{Q_{\lambda}}(|v_{\lambda}|^{p-1}v_{\lambda}-|v|^{p-1}v)(u_{\lambda}-u)^{+}\ dV_{\mathbb{H}^{N}}\\[2mm]
&\leq (\int_{Q_{\lambda}}|v_{\lambda}|^{p+1}\ dV_{\mathbb{H}^{N}})^{\frac{p-1}{p+1}}(\int_{Q_{\lambda}}|(v_{\lambda}-v)^{+}|^{p+1}
\ dV_{\mathbb{H}^{N}})^{\frac{1}{p+1}}(\int_{Q_{\lambda}}|(u_{\lambda}-u)^{+}|^{p+1}\ dV_{\mathbb{H}^{N}})^{\frac{1}{p+1}}\\[2mm]
&\leq c(\int_{Q_{\lambda}}|v_{\lambda}|^{p+1}\ dV_{\mathbb{H}^{N}})^{\frac{p-1}{p+1}}
(\int_{Q_{\lambda}}|\nabla_{\mathbb{H}^{N}}(v_{\lambda}-v)^{+}|^{2}\ dV_{\mathbb{H}^{N}})^{\frac{1}{2}}
(\int_{Q_{\lambda}}|\nabla_{\mathbb{H}^{N}}(u_{\lambda}-u)^{+}|^{2}\ dV_{\mathbb{H}^{N}})^{\frac{1}{2}},
\end{array}\]
and
\[\begin{array}{ll}
&\ \  \int_{Q_{\lambda}}|\nabla_{\mathbb{H}^{N}}(v_{\lambda}-v)^{+}|^{2}\ dV_{\mathbb{H}^{N}}\\[2mm]
&\leq c(\int_{Q_{\lambda}}|u_{\lambda}|^{q+1}\ dV_{\mathbb{H}^{N}})^{\frac{q-1}{q+1}}
(\int_{Q_{\lambda}}|\nabla_{\mathbb{H}^{N}}(u_{\lambda}-u)^{+}|^{2}\ dV_{\mathbb{H}^{N}})^{\frac{1}{2}}
(\int_{Q_{\lambda}}|\nabla_{\mathbb{H}^{N}}(v_{\lambda}-v)^{+}|^{2}\ dV_{\mathbb{H}^{N}})^{\frac{1}{2}}.
\end{array}\]
It implies that
\[\begin{array}{ll}
&\ \  (\int_{Q_{\lambda}}|\nabla_{\mathbb{H}^{N}}(u_{\lambda}-u)^{+}|^{2}\ dV_{\mathbb{H}^{N}})^{\frac{1}{2}}\\[2mm]
&\leq c(\int_{Q_{\lambda}}|v_{\lambda}|^{p+1}\ dV_{\mathbb{H}^{N}})^{\frac{p-1}{p+1}}
(\int_{Q_{\lambda}}|\nabla_{\mathbb{H}^{N}}(v_{\lambda}-v)^{+}|^{2}\ dV_{\mathbb{H}^{N}})^{\frac{1}{2}}.
\end{array}\]
and
\[\begin{array}{ll}
&\ \  (\int_{Q_{\lambda}}|\nabla_{\mathbb{H}^{N}}(v_{\lambda}-v)^{+}|^{2}\ dV_{\mathbb{H}^{N}})^{\frac{1}{2}}\\[2mm]
&\leq c(\int_{Q_{\lambda}}|u_{\lambda}|^{q+1}\ dV_{\mathbb{H}^{N}})^{\frac{q-1}{q+1}}
(\int_{Q_{\lambda}}|\nabla_{\mathbb{H}^{N}}(u_{\lambda}-u)^{+}|^{2}\ dV_{\mathbb{H}^{N}})^{\frac{1}{2}}.
\end{array}\]
Thus
\[[1-c(\int_{Q_{\lambda}}|u_{\lambda}|^{q+1}\ dV_{\mathbb{H}^{N}})^{\frac{q-1}{q+1}}
(\int_{Q_{\lambda}}|v_{\lambda}|^{p+1}\ dV_{\mathbb{H}^{N}})^{\frac{p-1}{p+1}}]
(\int_{Q_{\lambda}}|\nabla_{\mathbb{H}^{N}}(u_{\lambda}-u)^{+}|^{2}\ dV_{\mathbb{H}^{N}})^{\frac{1}{2}}\leq 0.\]
and
\[[1-c(\int_{Q_{\lambda}}|u_{\lambda}|^{q+1}\ dV_{\mathbb{H}^{N}})^{\frac{q-1}{q+1}}
(\int_{Q_{\lambda}}|v_{\lambda}|^{p+1}\ dV_{\mathbb{H}^{N}})^{\frac{p-1}{p+1}}]
(\int_{Q_{\lambda}}|\nabla_{\mathbb{H}^{N}}(v_{\lambda}-v)^{+}|^{2}\ dV_{\mathbb{H}^{N}})^{\frac{1}{2}}\leq 0.\]

Since we can choose $\lambda_{1}\in \mathbb{R}$, such that
\[c(\int_{Q_{\lambda}}|u_{\lambda}|^{q+1}\ dV_{\mathbb{H}^{N}})^{\frac{q-1}{q+1}}(\int_{Q_{\lambda}}|v_{\lambda}|^{p+1}\ dV_{\mathbb{H}^{N}})^{\frac{p-1}{p+1}}<1,\]
Then for any $\lambda>\lambda_{1}$, we have
\[\int_{Q_{\lambda}}|\nabla_{\mathbb{H}^{N}}(u_{\lambda}-u)^{+}|^{2}\ dV_{\mathbb{H}^{N}}\leq 0,\]
and
\[\int_{Q_{\lambda}}|\nabla_{\mathbb{H}^{N}}(v_{\lambda}-v)^{+}|^{2}\ dV_{\mathbb{H}^{N}}\leq 0.\]
Therefore $(u_{\lambda}-u)^{+}\equiv 0$,\,and  $(v_{\lambda}-v)^{+}\equiv 0$ in $Q_{\lambda}$,\, and $(\lambda_{1},+\infty)\subset \Lambda$.

Step 2. $\Lambda$ is bounded from below.

By $(u,v)\in H^{1}(\mathbb{H}^{N})\times H^{1}(\mathbb{H}^{N})$, we have
\[\lim\limits_{\lambda\rightarrow -\infty}\sup\limits_{Q_{\lambda}}u=\lim\limits_{\lambda\rightarrow -\infty}\sup\limits_{Q_{\lambda}}v=0,\]
and we may choose $\lambda_{2}$ such that
\[\sup\limits_{Q_{\lambda_{2}}}u<\frac{\sup\limits_{\mathbb{H}^{N}}u}{2},\ \ \,\sup\limits_{Q_{\lambda_{2}}}v<\frac{\sup\limits_{\mathbb{H}^{N}}v}{2}.\]
This implies that all $\lambda\in (-\infty,\lambda_{2})$ do not belong to $\Lambda$. Therefore $\Lambda$ is bounded from below, and we let $\bar{\Lambda}=\inf\Lambda$.

Step 3. $u\equiv u_{\bar{\Lambda}},\,v\equiv v_{\bar{\Lambda}}$, in $Q_{\bar{\Lambda}}$.

In fact, it is clear that by continuity of the foliation and of $(u,v)$, we have
\[u\geq u_{\bar{\Lambda}},\,v\geq v_{\bar{\Lambda}}, in \   Q_{\bar{\Lambda}}.\]

Now observing that if $u\equiv u_{\bar{\Lambda}}$, it follows from
\begin{equation}\label{eq:2.1}
\left\{
  \begin{array}{ll}
  \displaystyle
-\Delta_{\mathbb{H}^{N}}(u-u_{\bar{\Lambda}})=|v|^{p-1}v-|v_{\bar{\Lambda}}|^{p-1}v_{\bar{\Lambda}}      ,    \\
\displaystyle
-\Delta_{\mathbb{H}^{N}}(v-v_{\bar{\Lambda}})=|u|^{q-1}u-|u_{\bar{\Lambda}}|^{q-1}u_{\bar{\Lambda}}     ,    \\
\end{array}
\right.
\end{equation}
we can get $v\equiv v_{\bar{\Lambda}}$. So if we assume, by contradiction that the step 3 is not true, we conclude that
\begin{equation}\label{eq:2.2}
\left\{
  \begin{array}{ll}
  \displaystyle
u\geq u_{\bar{\Lambda}},\, v\geq v_{\bar{\Lambda}}      &\ \ {\rm for}\ x\in\  Q_{\bar{\Lambda}},
\\
\displaystyle
u\not\equiv u_{\bar{\Lambda}},\, v\not\equiv v_{\bar{\Lambda}}      ,
\end{array}
\right.
\end{equation}
By the strong maximum principle and connectedness of $Q_{\bar{\Lambda}}$ implying that
\begin{equation}\label{eq:2.3}
\left\{
  \begin{array}{ll}
  \displaystyle
u> u_{\bar{\Lambda}},      &\ \ {\rm for}\ x\in\  Q_{\bar{\Lambda}},    \\
\\
\displaystyle
v> v_{\bar{\Lambda}},      &\ \ {\rm for}\ x\in\  Q_{\bar{\Lambda}},    \\
\end{array}
\right.
\end{equation}
and
\begin{equation}\label{eq:2.4}
\left\{
  \begin{array}{ll}
  \displaystyle
X(u)(x)> 0,      &\ \ {\rm for}\ x\in\  U_{\bar{\Lambda}},    \\
\\
\displaystyle
X(v)(x)> 0,      &\ \ {\rm for}\ x\in\  U_{\bar{\Lambda}},    \\
\end{array}
\right.
\end{equation}
where $X$ is the killing vector field associated to the transformation group $A_{t}$, we shall see that this is impossible.

Choose $x_{1}\in U_{\bar{\Lambda}}$, and $R_{0}>0$. By the continuity of the foliation, there would exist $\varepsilon_{0}>0$, such that for $0<\varepsilon<\varepsilon_{0},\,I_{\bar{\Lambda}-\varepsilon}(B(x_{1},R_{0}))\subset B(x_{1},2R_{0})$. Moreover, by the definition of $\bar{\Lambda}$, we could construct an increasing sequence $\lambda_{n}>\bar{\lambda}$, such that $\lambda_{n}>\bar{\lambda}-\varepsilon_{0}$, and $\exists y_{n}\in Q_{\lambda_{n}}$, such that
\[u(y_{n})<u(I_{\lambda_{n}}(y_{n}))=u_{\lambda_{n}}(y_{n}),\]
or
\[v(y_{n})<v(I_{\lambda_{n}}(y_{n}))=v_{\lambda_{n}}(y_{n}).\]

We claim that $y_{n}\in B(x_{1},2R_{0})$. If it is not true taking $(u_{\lambda_{n}}-u)^{+}$ or $(v_{\lambda_{n}}-v)^{+}$ as test function(as in the first step). We would have that $u_{\lambda_{n}}\leq u$, or $v_{\lambda_{n}}\leq v$, in $Q_{\lambda_{n}}$. This proves our claim.

Modulo a subsequence. There would exist $y\in Q_{\bar{\Lambda}}$, such that $y_{n}\longrightarrow y$. By continuity we have
\[u(y)=\lim\limits_{n\rightarrow +\infty}u(y_{n})\leq\lim\limits_{n\rightarrow +\infty}u_{\lambda_{n}}(y_{n})=u_{\bar{\Lambda}}(y),\]
or
\[v(y)=\lim\limits_{n\rightarrow +\infty}v(y_{n})\leq\lim\limits_{n\rightarrow +\infty}v_{\lambda_{n}}(y_{n})=v_{\bar{\Lambda}}(y).\]
It implies that $y\in U_{\bar{\Lambda}}$. On the other hand, there exist points $\xi_{n}$ in the line segment between $y_{n}$ and $I_{\lambda_{n}}(y_{n})$, such that $X(u)(\xi_{n})\leq 0$, passing to the limit we should have $X(u)(y)\leq 0$. This is impossible $X(u)(x)> 0$, for all $x\in U_{\bar{\Lambda}}$. Hence $u\equiv u_{\bar{\Lambda}}$, and $v\equiv v_{\bar{\Lambda}}$, in $Q_{\bar{\Lambda}}$. $\Box$

\section  {Decay estimates}
The hyperbolic N-space $\mathbb{H}^N$, $N\geq 2$ is a complete simple connected Riemannian manifold
having constant sectional curvature equal to -1, and for a given dimensional number, any
two such spaces are isometric \cite{W}. There are several models for $\mathbb{H}^N$, the most important
being the half-space model, the ball model, and the hyperboloid or Lorentz model.

Let $\mathbb{B}^{N}=\{x\in \mathbb{R}^{N}:|x|<1\}$ denotes the unit disc in $\mathbb{R}^{N}$. The space $\mathbb{H}^{N}$ endowed with the Riemannian metric $g$ given by $g_{ij}=(\frac{2}{1-|x|^{2}})^{2}\delta_{ij}$ is called the ball model of the Hyperbolic space. The hyperbolic Laplacian $\Delta_{\mathbb{H}^{N}}$ is given by
\[\Delta_{\mathbb{H}^{N}}=(\frac{1-|x|^{2}}{2})^{2}\Delta+(N-2)\frac{1-|x|^{2}}{2}\langle x,\nabla\rangle,\]

Let $(u,v)$ be a positive symmetric solution of problem (\ref{eq:1.1}), and $u=u(|\xi|),\,v=v(|\xi|),\,|\xi|<1$, then
\begin{equation}\label{eq:3.1}
\left\{
  \begin{array}{ll}
  \displaystyle
(\frac{1-|\xi|^{2}}{2})^{2}\Delta u+(N-2)\frac{1-|\xi|^{2}}{2}\langle \xi,\nabla u\rangle+v^{p}=0,          \\
\\
\displaystyle
(\frac{1-|\xi|^{2}}{2})^{2}\Delta v+(N-2)\frac{1-|\xi|^{2}}{2}\langle \xi,\nabla v\rangle+u^{q}=0,         \\
\end{array}
\right.
\end{equation}

Setting $|\xi|=\tanh\frac{t}{2},\,u(t)=u(\tanh\frac{t}{2}),\,v(t)=v(\tanh\frac{t}{2}),\,k(t)=(\sinh t)^{N-1}$, it is easy to see that
\[\int_{\mathbb{H}^{N}}|u|^{q+1}\ dV_{\mathbb{H}^{N}}=w_{N-1}\int_{0}^{\infty}k(t)|u|^{q+1}dt,\]
\[\int_{\mathbb{H}^{N}}|v|^{p+1}\ dV_{\mathbb{H}^{N}}=w_{N-1}\int_{0}^{\infty}k(t)|v|^{p+1}dt,\]
\[\int_{\mathbb{H}^{N}}|\nabla_{\mathbb{H}^{N}}u|^{2}\ dV_{\mathbb{H}^{N}}=w_{N-1}\int_{0}^{\infty}k(t)|u'|^{2}dt,\]
\[\int_{\mathbb{H}^{N}}|\nabla_{\mathbb{H}^{N}}v|^{2}\ dV_{\mathbb{H}^{N}}=w_{N-1}\int_{0}^{\infty}k(t)|v'|^{2}dt,\]
where $w_{N-1}$ denotes the surface area of $S^{N-1}$.

In addition, (\ref{eq:3.1}) rewrites
\begin{equation}\label{eq:3.2}
\left\{
  \begin{array}{ll}
  \displaystyle
u''+\frac{N-1}{\tanh t}u'+v^{p}=0,          \\
\\
\displaystyle
v''+\frac{N-1}{\tanh t}v'+u^{q}=0,         \\
\\
\displaystyle
u'(0)=v'(0)=0.         \\
\end{array}
\right.
\end{equation}
as well as
\begin{equation}\label{eq:3.3}
\left\{
  \begin{array}{ll}
  \displaystyle
(k(t)u')'+k(t)v^{p}=0,          \\
\\
\displaystyle
(k(t)v')'+k(t)u^{q}=0,         \\
\\
\displaystyle
u'(0)=v'(0)=0
\end{array}
\right.
\end{equation}
and if $(u,v)\in H^{1}(\mathbb{H}^{N})\times H^{1}(\mathbb{H}^{N})$ solves (\ref{eq:3.1}), then
\[\int_{0}^{\infty}k(t)u'v'dt=\int_{0}^{\infty}k(t)u^{q+1}dt=\int_{0}^{\infty}k(t)v^{p+1}dt.\]

Now, let us notice that $(u,v)$ solves (\ref{eq:3.2}), and
\[J_{(u,v)}(t)=u'v'+\frac{|v|^{p+1}}{p+1}+\frac{|u|^{q+1}}{q+1},\]
then
\[\begin{array}{ll}
&\ \  \frac{d}{dt}J_{(u,v)}(t)\\[2mm]
&= u''v'+u'v''+v'v^{p}+u'u^{q}\\[2mm]
&=(u''+v^{p})v'+u'(v''+u^{q})\\[2mm]
&=-\frac{N-1}{\tanh t}v'^{2}--\frac{N-1}{\tanh t}u'^{2}\\[2mm]
&\leq 0.   \ \ \ \ \ \ \ \ \ \ \ \  \forall t> 0.
\end{array}\]

\begin{lemma}\label{lm:3.1}
Let $(u,v)\in H^{1}(\mathbb{H}^{N})\times H^{1}(\mathbb{H}^{N})$ be a positive solution of (\ref{eq:3.2}), then $u'(t)<0,\,v'(t)<0$, for every $t>0$, and
\[\lim\limits_{t\rightarrow +\infty}u(t)=\lim\limits_{t\rightarrow +\infty}v(t)=\lim\limits_{t\rightarrow +\infty}u'(t)=\lim\limits_{t\rightarrow +\infty}v'(t)=0.\]
\end{lemma}
\begin{proof}
By equation (\ref{eq:3.3}), we have $(k(t)u'(t))'<0$, and $(k(t)v'(t))'<0$, and from $u'(0)=v'(0)=0$, then $u'(t)<0,\,v'(t)<0,\,\forall t>0.$

In particular, by $u(t)>0,\,v(t)>0$, it exist
\[u(\infty)=\lim\limits_{t\rightarrow +\infty}u(t),\ \ \,v(\infty)=\lim\limits_{t\rightarrow +\infty}v(t).\]
Since $J_{(u.v)}$ is decreasing, then $\lim\limits_{t\rightarrow +\infty}u'(t)v'(t)$ exists.

Since $(u,v)\in H^{1}(\mathbb{H}^{N})\times H^{1}(\mathbb{H}^{N})$, then
\[\liminf\limits_{t\rightarrow +\infty}k(t)[u'^{2}(t)+u^{2}(t)]=0,\]
and
\[\liminf\limits_{t\rightarrow +\infty}k(t)[v'^{2}(t)+v^{2}(t)]=0.\]
Thus, we have that
\[\lim\limits_{t\rightarrow +\infty}u(t)=\lim\limits_{t\rightarrow +\infty}v(t)=\lim\limits_{t\rightarrow +\infty}u'(t)=\lim\limits_{t\rightarrow +\infty}v'(t)=0.\]
\end{proof}

\begin{lemma}\label{lm:3.2}
Let $(u,v)$ be a positive solution of problem (\ref{eq:1.1}), and $(u,v)\in H^{1}(\mathbb{H}^{N})\times H^{1}(\mathbb{H}^{N})$, then
\[\lim\limits_{t\rightarrow +\infty}\frac{\log u^{2}}{t}=\lim\limits_{t\rightarrow +\infty}\frac{\log v^{2}}{t}=\lim\limits_{t\rightarrow +\infty}\frac{\log u'^{2}}{t}=\lim\limits_{t\rightarrow +\infty}\frac{\log v'^{2}}{t}=-2(N-1).\]
\end{lemma}
\begin{proof}
By (\ref{lm:3.1}), we obtain
\[\lim\limits_{t\rightarrow +\infty}u(t)=\lim\limits_{t\rightarrow +\infty}v(t)=0.\]
Then it exists $t_{\varepsilon}>0$, such that
\[\coth t\leq 1+\varepsilon,\,v^{p}(t)\leq \varepsilon v(t),\,u^{q}(t)\leq \varepsilon u(t),\ \ \,\forall t\geq t_{\varepsilon}.\]

Since $u'(t)<0,\,v'(t)<0$, we have for $t\geq t_{\varepsilon}$,
\begin{equation}\label{eq:3.4}
\left\{
  \begin{array}{ll}
  \displaystyle
u''+(N-1)(1+\varepsilon)u'\leq u''+(N-1)\coth tu'+v^{p}=0,          \\
\\
\displaystyle
u''+(N-1)u'+\varepsilon v\geq u''+(N-1)\coth tu'+v^{p}=0.         \\
\end{array}
\right.
\end{equation}
and
\begin{equation}\label{eq:3.5}
\left\{
  \begin{array}{ll}
  \displaystyle
v''+(N-1)(1+\varepsilon)v'\leq v''+(N-1)\coth tv'+u^{q}=0,          \\
\\
\displaystyle
v''+(N-1)v'+\varepsilon u\leq v''+(N-1)\coth tv'+u^{q}=0.         \\
\end{array}
\right.
\end{equation}
Then, we get
\[(u+v)''+(N-1)(1+\varepsilon)(u+v)'\leq (u+v)''+(N-1)\coth t(u+v)'+u^{q}+v^{p}=0,\]
\[(u+v)'+(N-1)(u+v)'+\varepsilon(u+v)\geq (u+v)''+(N-1)\coth t(u+v)'+u^{q}+v^{p}=0.\]
It implies that
\begin{equation}\label{eq:3.6}
(u+v)''+(N-1)(1+\varepsilon)(u+v)'\leq (u+v)''+(N-1)(u+v)'+\varepsilon(u+v).
 \end{equation}

 Let $\mu^{-}(\varepsilon)=-(N-1)(1+\varepsilon),\,\mu^{+}(\varepsilon)=0$, and $\nu^{-}(\varepsilon)=\frac{-(N-1)-\sqrt{(N-1)^{2}-4\varepsilon}}{2},\,\nu^{+}(\varepsilon)=\frac{-(N-1)+\sqrt{(N-1)^{2}-4\varepsilon}}{2}$ be the characteristic roots of the differential Polinamials in the L.h.S and R.h.S of (\ref{eq:3.6}) respectively. We choose $\varepsilon<\frac{(N-1)^{2}}{4}$, then $\nu^{\pm}(\varepsilon)$ is real and distinct. Similar as\cite{MS}, we can get

\begin{equation}\label{eq:3.7}
\left\{
  \begin{array}{ll}
  \displaystyle
u(t)+v(t)\geq ([u(t_{\varepsilon})+v(t_{\varepsilon})e^{-\mu^{-}(\varepsilon)t_{\varepsilon}}])e^{\mu^{-}(\varepsilon)t},     \ \ \ \ \ \forall t\geq t_{\varepsilon} ,    \\
\\
\displaystyle
u(t)+v(t)\leq ([u(t_{\varepsilon})+v(t_{\varepsilon})e^{-\nu^{-}(\varepsilon)t_{\varepsilon}}])e^{\nu^{-}(\varepsilon)t},     \ \ \ \ \ \forall t\geq t_{\varepsilon} ,.         \\
\end{array}
\right.
\end{equation}

\begin{equation}\label{eq:3.8}
\left\{
  \begin{array}{ll}
  \displaystyle
u'(\tau)+v'(\tau)\geq \mu^{-}(\varepsilon)(u(\tau)+v(\tau)),     \ \ \ \ \ \forall \tau\geq t_{\varepsilon} ,    \\
\\
\displaystyle
u'(\tau)+v'(\tau)\leq \nu^{-}(\varepsilon)(u(\tau)+v(\tau)),     \ \ \ \ \ \forall \tau\geq t_{\varepsilon} ,.         \\
\end{array}
\right.
\end{equation}

We see from (\ref{eq:3.7}) that
\[2\mu^{-}(\varepsilon)\leq \liminf\limits_{t\rightarrow +\infty}\frac{\log(u+v)^{2}}{t}\leq \limsup\limits_{t\rightarrow +\infty}\frac{\log(u+v)^{2}}{t}\leq 2\nu^{-}(\varepsilon),\ \ \,\forall \varepsilon>0,\]
and hence
\begin{equation}\label{eq:3.9}
\lim\limits_{t\rightarrow +\infty}\frac{\log(u+v)^{2}}{t}=-2(N-1),
 \end{equation}
From (\ref{eq:3.7}) and (\ref{eq:3.8}), we also get
\begin{equation}\label{eq:3.10}
\lim\limits_{t\rightarrow +\infty}\frac{\log(u'+v')^{2}}{t}=-2(N-1).
 \end{equation}

Since $u'(t)<0,\,v'(t)<0$, we conclude from (\ref{eq:3.10}) that
\[\lim\limits_{t\rightarrow +\infty}\frac{\log u'^{2}}{t}=-2(N-1),\]
and
\[\lim\limits_{t\rightarrow +\infty}\frac{\log v'^{2}}{t}=-2(N-1).\]
From $u(t)>0,\,v(t)>0$, and (\ref{eq:3.9}), we also have that
\[\lim\limits_{t\rightarrow +\infty}\frac{\log u^{2}}{t}=\lim\limits_{t\rightarrow +\infty}\frac{\log v^{2}}{t}=-2(N-1).\]
\end{proof}

\textbf{Proof of Theorem \ref{tm:1.2}:}
By Lemma \ref{lm:3.1} and Lemma \ref{lm:3.2}, Theorem \ref{tm:1.1} is proved. $\Box$

\section  {Existence of radial solutions of (\ref{eq:1.1})}
Now, we denote the functional
\[I(z)=I(u,v)=\int_{\mathbb{H}^{N}}\nabla_{\mathbb{H}^{N}}u\cdot\nabla_{\mathbb{H}^{N}}v\ dV_{\mathbb{H}^{N}}
-\frac{1}{p+1}\int_{\mathbb{H}^{N}}(v_+)^{p+1}\ dV_{\mathbb{H}^{N}}-\frac{1}{q+1}\int_{\mathbb{H}^{N}}(u_+)^{q+1}\ dV_{\mathbb{H}^{N}},\]
which is the functional of problem (\ref{tm:1.1}), where $z=(u, v)$, $u_+=\max\{u, 0\}, v_+=\max\{v,0\}$.

Observe that the quadratic part of $I$ is well-defined, if $u,\,v\in H^{1}(\mathbb{H}^{N})$. But, if we take such $u$ and $v$, the nonlinear part is well defined if $p$ and $q$ are both less or equal to $\frac{N+2}{N-2}$, for $N\geq 3$. However, we would like to consider pairs $(p,q)$ that do not satisfy this restriction. The basic requirement would be that $(p,q)$ is below the critical hyperbola, one of them could be large that $\frac{N+2}{N-2}$. So, we need inhomogeneous  Sobolev spaces on hyperbolic space.

Now, we will introduce some aspects of the harmonic analysis and the geometry of hyperbolic space, see \cite{B} \cite{H3}\cite{H1}\cite{H2}\cite{T1}  and reference therein.

We consider the Minkowski space $\mathbb{R}^{N+1}$ with the standard Minkowski metric
\[(dg)^{2}=-(dx_{0})^{2}+(dx_{1})^{2}+\cdots+(dx_{N})^{2},\]
and defined the bilinear form on $\mathbb{R}^{N+1}\times \mathbb{R}^{N+1}$
\[[x,y]=x_{0}y_{0}-x_{1}y_{1}-\cdots-x_{N}y_{N}.\]
The Hyperbolic space $\mathbb{H}^{N}$ is defined as a subset of $\mathbb{R}^{N+1}$ by
\[\mathbb{H}^{N}=\{x\in\mathbb{R}^{N+1}\mid [x,x]=1,\,\ and  \,\,  x_{0}>0\}\]
which is the  hyperboloid  model.

In geodesic polar coordinates, the Riemnanian metric is given by
\[(dg)^{2}=(dr)^{2}+(\sinh r)^{2}(dg_{\mathbb{S}^{N-1}})^{2},\]
and the Riemnanian volume by
\[\ dV_{\mathbb{H}^{N}}=(\sinh r)^{N-1}drdV_{\mathbb{S}^{N-1}}.\]

The Fourier transform (as defined by Helgason \cite{H3}   ) takes suitable functions defined on $\mathbb{H}^{N}$ to functions defined on $\mathbb{R}\times\mathbb{S}^{N-1}$. For $w\in\mathbb{S}^{N-1}$, and $\lambda\in\mathbb{C}$, let $b(w)=(1,w)\in\mathbb{R}^{N+1}$, and
\[h_{\lambda,w}:\mathbb{H}^{N}\longrightarrow\mathbb{C},\,\,h_{\lambda,w}(x)=[x,b(w)]^{i\lambda-d},\]
where $d=\frac{N-1}{2}$. It is known that
\begin{equation}\label{eq:4.1}
\Delta_{\mathbb{H}^{N}}h_{\lambda,w}=-(\lambda^{2}+d^{2})h_{\lambda,w}.
 \end{equation}

The Fourier transform of $f\in C_{0}^{\infty}(\mathbb{H}^{N})$ is defined by the formula
\[\tilde{f}(\lambda,w)=F(f)=\int_{\mathbb{H}^{N}}f(x)[x,b(w)]^{i\lambda-d}\ dV_{\mathbb{H}^{N}}.\]
This transformation admits a Fourier inversion formula, if $f\in C_{0}^{\infty}(\mathbb{H}^{N})$, then
\[f(x)=F^{-1}(\tilde{f})=\int_{0}^{\infty}\int_{\mathbb{S}^{N-1}}\tilde{f}(\lambda,w)[x,b(w)]^{-i\lambda-d}|C(\lambda)|^{-2}d\lambda dw,\]
where, for a suitable constant $C$,
\[C(\lambda )=C\frac{\Gamma(i\lambda)}{\Gamma(p+i\lambda)},\]
is the Harish-chardrac-function on $\mathbb{H}^{d}$, and the invariant measure of $\mathbb{S}^{N-1}$ is normalized to $1$.

Now we define the inhomogeneous Sobolev space on $\mathbb{H}^{N}$. There are two possible definitions: using the Riemannian structure or using the Fourier transform. There two definitions agree\cite{T}. For $p \in (1,\infty)$ and $s\in\mathbb{R}$, we defined the Sobolev space $H^{s,p} (\mathbb{H}^{N})$ as the closure of $C_{0}^{\infty}(\mathbb{H}^{N})$ under the norm
\[\|f\|_{H^{s,p}(\mathbb{H}^{N})}=\|(-\Delta_{\mathbb{H}^{N}})^{\frac{s}{2}}f\|_{L^{p}(\mathbb{H}^{N})}=\|F^{-1}(\lambda^{2}+d^{2})^{\frac{s}{2}}F(u)\|_{L^{p}(\mathbb{H}^{N})}.\]

For $s\in\mathbb{R}$, let $H^s(\mathbb{H}^N)=H^{s,2} (\mathbb{H}^{N})$,  in particular, $W^{1,p} (\mathbb{H}^{N})=H^{1,p} (\mathbb{H}^{N})$, and the Sobolev embedding theorem
\[H^{s,p}\hookrightarrow L^{q},\,\,if  \,\,1<p\leq q<\infty,\,\,\,and \,\, s=\frac{N}{p}-\frac{N}{q}\]
holds. Moreover, for $s>t, H^{s}\subset H^{t}$, see \cite{T1}. Let
\[H_{r}^{s}(\mathbb{H}^{N})=\{u\in H^{s}(\mathbb{H}^{N}):u \,\,is\,\, radial\},\]
we have the following imbedding theorem. The case $s=1$ was proved by \cite{BS1}.

\begin{lemma}\label{lm:4.1}
Let $s>0$, then the restriction to $H_{r}^{s}(\mathbb{H}^{N})$ of the Sobolev imbedding of $H^{s}(\mathbb{H}^{N})$ into $L^{r}(\mathbb{H}^{N})$ is continuous, if $2\leq r\leq \frac{2N}{N-2s}$, and it is compact if $2< r<\frac{2N}{N-2s}$.
\end{lemma}
\begin{proof}
From \cite{BS1}, we have that $H_{r}^{1}(\mathbb{H}^{N})\hookrightarrow L^{p}(\mathbb{H}^{N}), 2<p<\frac{2N}{N-2}$ is compact.

Case 1. For $s>1$.

From the definition of $H^{s}(\mathbb{H}^{N})$, we have
\[H_{r}^{s}(\mathbb{H}^{N})\hookrightarrow H_{r}^{1}(\mathbb{H}^{N}).\]
Hence,
\[H_{r}^{s}(\mathbb{H}^{N})\hookrightarrow L^{p}(\mathbb{H}^{N}),\,\, 2<p<\frac{2N}{N-2} \,\, is \,\, compact.\]
For $\frac{2N}{N-2}\leq \gamma<\frac{2N}{N-2s}$, we can deduce this Lemma using Sobolev inequalities and H\"{o}lder inequalities. Indeed, if $\frac{2N}{N-2}\leq r<\frac{2N}{N-2s}$, we get
\[\int_{\mathbb{H}^{N}}|u|^{\gamma}\ dV_{\mathbb{H}^{N}}\leq(\int_{\mathbb{H}^{N}}|u|^{\frac{2N}{N-2s}}\ dV_{\mathbb{H}^{N}})^{a}(\int_{\mathbb{H}^{N}}|u|^{p}\ dV_{\mathbb{H}^{N}})^{1-a},\]
where $a=\frac{2N-(N-2s)\gamma}{2N-p(N-2s)}$.

Case 2. For $0<s<1$.

From \cite{T1}, we have that
\[[L^{2}(\mathbb{H}^{N}),H^{1}(\mathbb{H}^{N})]_{s}=H^{s}(\mathbb{H}^{N}),\]
and
\[[L_{r}^{2}(\mathbb{H}^{N}),H_{r}^{1}(\mathbb{H}^{N})]_{s}=H_{r}^{s}(\mathbb{H}^{N}).\]

From section 2, setting $|\xi|=\tanh \frac{t}{2},\,\,u(t)=u(\tanh \frac{t}{2}),\,\,k(t)=(\sinh t)^{N-1}$, we have
\[\int_{\mathbb{H}^{N}}|u|^{q}\ dV_{\mathbb{H}^{N}}=w_{N-1}\int_{0}^{\infty}k(t)|u|^{q}dt,\]
\[\int_{\mathbb{H}^{N}}|\nabla_{\mathbb{H}^{N}}u|^{2}\ dV_{\mathbb{H}^{N}}=w_{N-1}\int_{0}^{\infty}k(t)|u'|^{2}dt.\]

Now we claim that:
\[\|(\sinh t)^{\frac{N-1}{2}}u(t)\|_{H^{1} (1,\infty)}\leq C\|u\|_{H^{1} (\mathbb{H}^{N})},\,\,\forall u\in H_{r}^{1}(\mathbb{H}^{N}).\]

To do this, let $v(t)=(\sinh t)^{\frac{N-1}{2}}u(t)$, we have
\[\begin{array}{ll}
|\frac{dv}{dt}|&=|\frac{N-1}{2}(\sinh t)^{\frac{N-3}{2}}(\cosh t) u(t)+(\sinh t)^{\frac{N-1}{2}}u'(t)|\\[2mm]
 &\leq C \frac{1}{\tanh t}|v|+(\sinh t)^{\frac{N-1}{2}}|\frac{du}{dt}|.
 \end{array}\]
It implies that
\[\begin{array}{ll}
\int_{1}^{\infty}|\frac{dv}{dt}|^{2}dt&\leq C(\int_{0}^{\infty}(\sinh t)^{N-1}u^{2}(t)dt+\int_{0}^{\infty}(\sinh t)^{N-1}|\frac{du}{dt}|^{2}dt)\\[2mm]
 &\leq C\|u\|_{H^{1}(\mathbb{H}^{N})}.
 \end{array}\]
For
\[\int_{0}^{\infty}(\sinh t)^{N-1}u^{2}(t)dt=w_{N-1}^{-1}\|u\|_{L^2(\mathbb{H}^{N})},\]
and by the interpolate theory, we get
\[\|(\sinh t)^{\frac{N-1}{2}}u(t)\|_{H^{s}(1,\infty)}\leq C\|u\|_{H_{r}^{s}(\mathbb{H}^{N})}.\]
Thus, we can easily deduce this Lemma by $\forall  q\in(2,\frac{2N}{N-2s})$,
\[\|u\|_{L^{q}(\mathbb{H}^{N})}\leq C\|u\|_{H^{s}(\mathbb{H}^{N})},\,\,\,\forall u\in H_{r}^{s}(\mathbb{H}^{N}),\]
\[\|(\sinh t)^{\frac{N-1}{2}}u(t)\|_{L^{q}(1,\infty)}\leq C\|u\|_{H^{s}(\mathbb{H}^{N})}.\]
Indeed, we have
\[\begin{array}{ll}
\int_{R}^{\infty}(\sinh t)^{N-1}|u(t)|^{q}dt&\leq\frac{1}{(\frac{e^{R}+e^{-R}}{2})^{(N-1)(\frac{q}{2}-1)}}\int_{R}^{\infty}(\sinh t)^{\frac{(N-1)q}{2}}|u(t)|^{q}dt\\[2mm]
 &\leq\frac{\mathcal{C}}{(\frac{e^{R}+e^{-R}}{2})^{(N-1)(\frac{q}{2}-1)}}\|u\|_{H^{s}(\mathbb{H}^{N})}.
 \end{array}\]
It implies that if $R$ is large enough,
\[w_{N-1}\int_{R}^{\infty}(\sinh t)^{N-1}|u(t)|^{q}dt\leq \varepsilon,\,\,\,\forall u\in H_{r}^{s}(\mathbb{H}^{N})\]

\end{proof}

Now, let $L_{r}^{2}(\mathbb{H}^{N})$ be the space of $L^{2}$-functions in $\mathbb{H}^{N}$ which are radially symmetric. Let $T=-\Delta_{\mathbb{H}^{N}}$ with the domain $D(T)=H_{r}^{2}(\mathbb{H}^{N})$ which is the space of radial symmetric functions that are in $L^{2}$ and have second derivatives in $L^{2}$. For $0\leq s\leq 2$, the space $E^{s}$, which is the domain $D(T^{\frac{s}{2}})$, is precisely the space obtained by interpolation between $H_{r}^{2}(\mathbb{H}^{N})$ and $L_{r}^{2}(\mathbb{H}^{N})$,
\[[H_{r}^{2}(\mathbb{H}^{N}),L_{r}^{2}(\mathbb{H}^{N})]_{1-\frac{s}{2}}.\]

In this case, the space $E^{s}$ is the usual Sobolev space $H_{r}^{s}(\mathbb{H}^{N})$. So denoting by $A=(-\Delta_{\mathbb{H}^N})^{\frac{1}{2}}$, we have for all $0\leq s\leq 2$,
\[D(A^s)=H_{r}^{s}(\mathbb{H}^{N}).\]

Let $E=H_{r}^{s}(\mathbb{H}^{N})\times H_{r}^{t}(\mathbb{H}^{N})$  and bilinear form: $B:E\times E\longrightarrow \mathbb{R}$ is define by
\[B[(u,v),(\Phi,\Psi)]=\int_{\mathbb{H}^{N}}A^{s}u A^{t}\Psi+A^{s}\Phi A^{t}v,\]
and the corresponding quadratic form by
\[Q(z)=\int_{\mathbb{H}^{N}}A^{s}u A^{t}v,\,\,\,(u,v)\in E.\]
This quadratic form with replace $\int_{\mathbb{H}^{N}}\nabla_{\mathbb{H}^{N}}u\cdot\nabla_{\mathbb{H}^{N}}v$, so we consider the functional $\Phi:E\longrightarrow \mathbb{R}$, defined by
\[Q(z)=\int_{\mathbb{H}^{N}}A^{s}u A^{t}v\ dV_{\mathbb{H}^{N}}-\frac{1}{p+1}\int_{\mathbb{H}^{N}}(v_+)^{p+1}\ dV_{\mathbb{H}^{N}}-\frac{1}{q+1}\int_{\mathbb{H}^{N}}(u_+)^{q+1}\ dV_{\mathbb{H}^{N}},\]
For $z=(u,v)\in E, \Phi$ is a $C^{1}$ functional and
\[\langle \Phi'(z),y \rangle_{E}=\int_{\mathbb{H}^{N}}A^{s}u A^{t}\Psi+A^{s}\Phi A^{t}v-\int_{\mathbb{H}^{N}}(v_+)^{p}\Phi-\int_{\mathbb{H}^{N}}(u_+)^{q}\Psi,\]
For $z=(u,v)\in E$ and $y=(\Phi,\Psi)\in E$. So the critical points of $\Phi$ satisfy the equations
\[\int_{\mathbb{H}^{N}}A^{s}u A^{t}\Psi-\int_{\mathbb{H}^{N}}(v_+)^{p}\Psi=0,\,\,\, for\,\, all\,\, \Psi\in H_{r}^{t}(\mathbb{H}^{N}),\]
and
\[\int_{\mathbb{H}^{N}}A^{s}\Phi A^{t}v-\int_{\mathbb{H}^{N}}(u_+)^{q}\Phi=0,\,\,\, for\,\, all\,\, \Phi\in H_{r}^{s}(\mathbb{H}^{N}).\]

Similar as \cite{FJ}, we have $E=E^-\oplus E^+$ and  $B[z^+,z^-]=0$ for $z^+\in E^+, z^-\in E^-$, where
\[E^-=\{(u, -A^{-t}A^su):u\in H^s_r(\mathbb{H}^N)\}, \quad E^+=\{(u, A^{-t}A^su):u\in H^s_r(\mathbb{H}^N)\}.\]
We also have $Q(z)=\frac{1}{2}B[z,z]$ and $\frac{1}{2}\|z\|_E^2=Q(z^+)-Q(z^-)$, where $z=(u, v)\in E$ and $z=z^++z^-, z^+\in E^+, z^-\in E^-$.

\textbf{Proof of Theorem \ref{tm:1.3}}

Thanks to $H_{r}^{s}(\mathbb{H}^{N})\hookrightarrow L^{q}(\mathbb{H}^{N}),\,\,2<q<\frac{2N}{N-2s}$ is compact, then $\Phi:E\longrightarrow \mathbb{R}$ satisfies the Palais-Smale conditions as \cite{FJ} and using the linking theorem similarly as \cite{FJ}, we can get a radial solutions of problem (\ref{eq:1.1}).

\section  {Existence of Ground state solutions}

By Theorem  \ref{tm:1.3} , we know that the set
\[\{(u,v)\in E=H^{1}(\mathbb{H}^{N})\times H^{1}(\mathbb{H}^{N}): (u,v)\ \ is\ \ a\ \ nontrivial\ \ solution\ \ of\ \ (\ref{eq:1.1})\},\]
is non-empty set. We defined
\[I^{\infty}=\inf\{I(u,v)|(u,v)\ \ is\ \ a\ \ nontrivial\ \ solution\ \ of\ \ (\ref{eq:1.1})\}.\]
In order to give the proof of Theorem \ref{tm:1.4}, let us define
\begin{definition}
For $r>0$, define $S_r=\{x\in \mathbb{R}^N:|x|^2=1+r^2\}$ and for $a\in S_r$ define
\[A(a, r)=B(a, r)\cap \mathbb{H}^N\]
where $B(a, r)$ is the open ball in the Euclidean space with center $a$ and radius $r>0$.  Moreover, for the choice of $a$ and $r$,
 $\partial B(a, r)$ is orthogonal to $S^{N-1}$.
\end{definition}
Similarly as \cite{BS1}, we have
\begin{lemma}\label{lm:3.5}
Let $r_1>0, r_2>0$ and $A(a_i, r_i), i=1,2$ be as in the above definition, then there exists $\tau\in \mathcal {I}(\mathbb{H}^N)$
 such that $\tau(A(a_1, r_1))=A(a_2, r_2)$, where $\mathcal {I}(\mathbb{H}^N)$ is the isometry group of $\mathbb{H}^N$.
\end{lemma}
\begin{lemma}\label{lm:5.1}
If $p,q<\frac{N+2}{N-2}$, then $I^{\infty}$ is attained and $I^{\infty}>0$.
\end{lemma}
\begin{proof}
By Theorem  \ref{tm:1.3} , there exists a positive solution of (\ref{eq:1.1}), so
$\{z\in E:I'(z)=0,z\neq 0\}\neq \emptyset$, and $I^{\infty}$ is finite. If $z=(u,v)$  is a solution of (\ref{eq:1.1}), then
\[2\int_{\mathbb{H^{N}}}\nabla_{\mathbb{H^{N}}} u\cdot\nabla_{\mathbb{H^{N}}} v\ dV_{\mathbb{H}^{N}}=\int_{\mathbb{H^{N}}}(u_+)^{q+1}+(v_+)^{p+1}\ dV_{\mathbb{H}^{N}},\]
and hence
\[\begin{array}{ll}
I(z)&=I(u,v)\\[2mm]
  &=\int_{\mathbb{H^{N}}}\nabla_{\mathbb{H^{N}}} u\cdot\nabla_{\mathbb{H^{N}}} v\ \ dV_{\mathbb{H}^{N}}-\frac{1}{p+1}\int_{\mathbb{H^{N}}}(v_+)^{p+1}\ dV_{\mathbb{H}^{N}}-\frac{1}{q+1}\int_{\mathbb{H^{N}}}(u_+)^{q+1}\ dV_{\mathbb{H}^{N}}\\[2mm]
  &=(\frac{1}{2}-\frac{1}{p+1})\int_{\mathbb{H^{N}}}(v_+)^{p+1}\ dV_{\mathbb{H}^{N}}+(\frac{1}{2}-\frac{1}{q+1})\int_{\mathbb{H^{N}}}(u_+)^{q+1}\ dV_{\mathbb{H}^{N}}\\[2mm]
  &\geq 0.
 \end{array}\]

Now, we show that $I^{\infty}$ is obtained and positive.

Step 1. We show that the set of non-trivial solutions is bounded from below. In fact, we have
\begin{equation}\label{eq:5.1}
\begin{array}{ll}
\|(u)_+\|_{H^{1}(\mathbb{H}^{N})}^{2}&=\int_{\mathbb{H}^{N}}(v_+)^p u_+\ dV_{\mathbb{H}^{N}}\\[2mm]
&\leq (\int_{\mathbb{H}^{N}}(v_+)^{p+1}\ dV_{\mathbb{H}^{N}})^{\frac{p}{p+1}}(\int_{\mathbb{H}^{N}}(u_+)^{p+1}\ dV_{\mathbb{H}^{N}})^{\frac{1}{p+1}}\\[2mm]
&\leq C \|v_+\|_{L^{p+1}(\mathbb{H}^{N})}^{p}\|u_+\|_{H^{1}(\mathbb{H}^{N})}.
\end{array}
 \end{equation}
It implies that
\[\|u_+\|_{H^{1}(\mathbb{H}^{N})}\leq C\|v_+\|^{p}_{L^{p+1}(\mathbb{H}^{N})}.\]

Now, using the two equations, we obtain
\begin{equation}\label{eq:5.2}
\int_{\mathbb{H}^{N}}(v_+)^{p+1}\ dV_{\mathbb{H}^{N}}=\int_{\mathbb{H^{N}}}\nabla_{\mathbb{H^{N}}}u\cdot\nabla_{\mathbb{H^{N}}} v\ dV_{\mathbb{H}^{N}}=\int_{\mathbb{H}^{N}}(u_+)^{q+1}\ dV_{\mathbb{H}^{N}}.
\end{equation}
From (\ref{eq:5.1}),(\ref{eq:5.2}), it follows that
\[\|u_+\|_{L^{q+1}(\mathbb{H}^{N})}\leq\|u_+\|_{H^{1}(\mathbb{H}^{N})}\leq C\|v_+\|_{L^{p+1}(\mathbb{H}^{N})}^{p}\leq C\|u_+\|_{L^{q+1}(\mathbb{H}^{N})}^{\frac{p(q+1)}{p+1}}.\]
Thus we get
\[\|u_+\|_{L^{q+1}(\mathbb{H}^{N})}\leq C\|u_+\|_{L^{q+1}(\mathbb{H}^{N})}^{\frac{p(q+1)}{p+1}}.\]
Which implies $\|u_+\|_{L^{q+1}(\mathbb{H}^{N})}\geq C$, since $\frac{p(q+1)}{p+1}>1$.

Similarly, we prove that $\|v_+\|_{L^{p+1}(\mathbb{H}^{N})}\geq C$.

Step 2. Suppose now that $z_{n}=(u_{n},v_{n})$ is a minimizing sequence of $I^{\infty}$, that is
\[I(z_{n})\longrightarrow I^{\infty},\,\,I'(z_{n})=0,\,\,z_{n}\neq 0.\]

Clearly, $\{z_{n}\}$ is a $(PS)_{I^{\infty}}$ sequence for $I$. Now, we want to prove that $\{z_{n}\}$ is uniformly bounded in $E$. Indeed, on one hand, we have
\[(\frac{1}{2}-\frac{1}{p+1})\int_{\mathbb{H}^{N}}[(v_{n})_+]^{p+1}\ dV_{\mathbb{H}^{N}}+(\frac{1}{2}-\frac{1}{q+1})
\int_{\mathbb{H}^{N}}[(u_{n})_+]^{q+1}\ dV_{\mathbb{H}^{N}}\leq I^{\infty},\]
it implies that
\[\int_{\mathbb{H}^{N}}[(v_{n})_+]^{p+1}\ dV_{\mathbb{H}^{N}}+\int_{\mathbb{H}^{N}}[(u_{n})_+]^{q+1}\ dV_{\mathbb{H}^{N}}\leq C.\]
On the other hand,
\[\begin{array}{ll}
0&=\langle I'(z_{n}),z_{n}^{+} \rangle\\[2mm]
&=\|z_{n}^{+}\|^{2}-\int_{\mathbb{H}^{N}}[(v_{n})_+]^{p} v_{n}^{+}\ dV_{\mathbb{H}^{N}}-\int_{\mathbb{H}^{N}}[(u_{n})_+]^{q}u_{n}^{+}\ dV_{\mathbb{H}^{N}}\\[2mm]
&\geq\|z_{n}^{+}\|^{2}-(\int_{\mathbb{H}^{N}}[(v_{n})_+]^{p+1}\ dV_{\mathbb{H}^{N}})^{\frac{p}{p+1}}(\int_{\mathbb{H}^{N}}|v_{n}^{+}|^{p+1}\ dV_{\mathbb{H}^{N}})^{\frac{1}{p+1}}\\[2mm]
&\ \ \
  -(\int_{\mathbb{H}^{N}}[(u_{n})_+]^{q+1}\ dV_{\mathbb{H}^{N}})^{\frac{q}{q+1}}(\int_{\mathbb{H}^{N}}|u_{n}^{+}|^{q+1}\ dV_{\mathbb{H}^{N}})^{\frac{1}{q+1}}\\[2mm]
&\geq\|z_{n}^{+}\|^{2}-(\int_{\mathbb{H}^{N}}[(v_{n})_+]^{p+1}\ dV_{\mathbb{H}^{N}})^{\frac{p}{p+1}}\|z_{n}^{+}\|\\[2mm]
&\ \ \
  -(\int_{\mathbb{H}^{N}}[(u_{n})_+]^{q+1}\ dV_{\mathbb{H}^{N}})^{\frac{q}{q+1}}\|z_{n}^{+}\|.
\end{array}\]
Which implies that $\|z_{n}^{+}\|\leq C$, similarly  $\|z_{n}^{-}\|\leq C$, thus $\|z_{n}\|=\|z_{n}^{+}\|+\|z_{n}^{-}\|\leq C$.

Hence, we may assume
\[z_{n}\longrightarrow z=(u,v)\,\,in \,\,E,\,\,\,z_{n}\longrightarrow z \,\,\,in \,\,\,L_{loc}^{q}(\mathbb{H}^{N})\times L_{loc}^{q}(\mathbb{H}^{N}),\]
as $n\longrightarrow\infty$ for any $2\leq q<\frac{2N}{N-2}$.

Step 3. By step 1, we know that, there exists $\delta'>0$ such that
\[\liminf\limits_{n\longrightarrow\infty}\int_{\mathbb{H}^{N}}[(u_{n})_+]^{q+1}\ dV_{\mathbb{H}^{N}}=\liminf\limits_{n\longrightarrow \infty}\int_{\mathbb{H}^{N}}[(v_{n})_+]^{p+1}\ dV_{\mathbb{H}^{N}}>\delta'>0.\]

Let us fix $\delta>0$, such that $0<\delta_{1}<\delta'<S_{N,p}^{\frac{p+1}{p-1}},\,\,0<\delta_{2}<\delta'<S_{N,p}^{\frac{q+1}{q-1}}$ , where $S_{N,P}$
satisfies that
\[S_{N,p}(\int_{\mathbb{H}^{N}}|u|^{p+1}\ dV_{\mathbb{H}^{N}})^\frac{2}{p+1}
\leq \int_{\mathbb{H}^{N}}|\nabla_{\mathbb{H}^N} u|^2\ \ dV_{\mathbb{H}^{N}},
\, \forall u\in H^1(\mathbb{H}^N), 1<p<\frac{N+2}{N-2}, N\geq 3.\]
 Let us define the concentration function:$Q_{n}:(0,\infty)\longrightarrow \mathbb{R}$ as follows.
\[Q_{n}(r)=\sup\limits_{x\in S_{r}}\int_{A(x,r)}[(u_{n})_+]^{q+1}\ dV_{\mathbb{H}^{N}}.\]

Now, $\lim\limits_{r\longrightarrow 0}Q_{n}(r)=0$, and $\lim\limits_{r\longrightarrow \infty}Q_{n}(r)>\delta_{2}$ as for large $r$. $A(x,r)$ approximates the intersection of $\mathbb{H}^{N}$ with a half space $\{y\in \mathbb{R}^{N}:(y,x)>0\}$. Therefore, we can choose a sequence $R_{n}>0$ and $x_{n}\in S_{R_{n}}$ such that
\[\sup\limits_{x\in S_{R_{n}}}\int_{A(x,R_{n})}[(u_{n})_+]^{q+1}\ dV_{\mathbb{H}^{N}}=\int_{A(x_{n},R_{n})}[(u_{n})_+]^{q+1}\ dV_{\mathbb{H}^{N}}=\delta_{2}.\]
For $x_{0}\in S_{\sqrt{3}}$, and using Lemma   , choosing $T_{n}\in I(\mathbb{H}^{N})$ such that
\[A(x_{n},R_{n})=T_{n}(A(x_{0},\sqrt{3})).\]

Now define $z_{n}'=(u_{n}',v_{n}')=(u_n\circ T_{n}(x),v_n\circ T_{n}(x))$. Since $T_{n}$ is an isometry one can easily see that $\{(u_{n}',v_{n}')\}$ is a $(PS)_{c}$ sequence of $I$ at the same level $I^{\infty}$ as $(u_{n},v_{n})$, and
\[\int_{\mathbb{H}^{N}}[(u_{n}')_+]^{q+1}\ dV_{\mathbb{H}^{N}}=\int_{\mathbb{H}^{N}}[(v_{n}')_+]^{p+1}\ dV_{\mathbb{H}^{N}}.\]
We also have
\begin{equation}\label{eq:5.3}
\begin{array}{ll}
\int_{A(x_{0},\sqrt{3})}[(u_{n}')_+]^{q+1}\ dV_{\mathbb{H}^{N}}=\int_{A(x_{n},R_{n})}[(u_{n}')_+]^{q+1}\ dV_{\mathbb{H}^{N}}=\sup\limits_{x\in S_{\sqrt{3}}}\int_{A(x,\sqrt{3})}[(u_{n}')_+]^{q+1}\ dV_{\mathbb{H}^{N}}=\delta_{2}.\\
\int_{A(x_{0},\sqrt{3})}[(v_{n}')_+]^{p+1}\ dV_{\mathbb{H}^{N}}=\int_{A(x_{n},R_{n})}|[(v_{n}')_+]^{p+1}\ dV_{\mathbb{H}^{N}}\leq \sup\limits_{x\in S_{r}}\int_{A(x,\sqrt{3})}[(v_{n}')_+]^{p+1}\ dV_{\mathbb{H}^{N}}=\delta_{1}.
 \end{array}
 \end{equation}
and $\|z_{n}\|=\|z_{n}'\|$. Therefore, up to a subsequence, we may assume
\[z_{n}'\longrightarrow z'=(u',v'),\,\,\,in  \  H^{1}(\mathbb{H}^{N})\times H^{1}(\mathbb{H}^{N}),\]
\[z_{n}'\longrightarrow z'=(u',v'),\,\,\,in  \  L_{loc}^{q}(\mathbb{H}^{N})\times L_{loc}^{q}(\mathbb{H}^{N}),\,\,2\leq q<\frac{N+2}{N-2}.\]
Moreover $z'=(u',v')$ solves the equation (\ref{eq:1.1}).

Now, we want to prove that $z'\neq 0$. If it is not true, we have $z'=(u',v')=(0,0)$. We claim that:

Claim: For any $1>r>2-\sqrt{3}$,
\[\int_{\mathbb{H}^{N}\bigcap\{|x|\geq r\}}[(u_{n}')_+]^{q+1}\ dV_{\mathbb{H}^{N}}=o(1),\,\,\,
\int_{\mathbb{H}^{N}\bigcap\{|x|\geq r\}}[(v_{n}')_+]^{p+1}\ dV_{\mathbb{H}^{N}}=o(1).\]

To do this, let us fix a point $a\in S_{\sqrt{3}}$. Let $\Phi\in C_{0}^{\infty}(A(a,\sqrt{3}))$ such that $0\leq\Phi\leq 1$, where $A(a,\sqrt{3})=B(a,\sqrt{3})\bigcap\mathbb{H}^{N}$, where $B(a,\sqrt{3})$ is the Euclidean ball with center $a$ and radius $\sqrt{3}$.
\[\int_{\mathbb{H}^{N}}\nabla_{\mathbb{H}^{N}} u_{n}'\nabla_{\mathbb{H}^{N}} \Psi \ dV_{\mathbb{H}^{N}}=\int_{\mathbb{H}^{N}}[(v_{n}')_+]^{p} \Psi \ dV_{\mathbb{H}^{N}},\,\,\,for\,\,  \Psi\in H^{1}(\mathbb{H}^{N}).\]

Now putting $\Psi=\Phi^{2}(u_{n}')_+$, in the above identity, we get
\[\int_{\mathbb{H}^{N}}\nabla_{\mathbb{H}^{N}} u_{n}'\nabla_{\mathbb{H}^{N}} [\Phi^{2}(u_{n}')_+] \ dV_{\mathbb{H}^{N}}=\int_{\mathbb{H}^{N}}[(v_{n}')_+]^{p} \Phi^{2}( u_{n}')_+ \ dV_{\mathbb{H}^{N}}.\]

A simple computation gives
\begin{equation}\label{eq:5.4}
\int_{\mathbb{H}^{N}}|\nabla_{\mathbb{H}^{N}} [\Phi (u_{n}')_+]|^{2} \ dV_{\mathbb{H}^{N}}=\int_{\mathbb{H}^{N}}[(v_{n}')_+]^{p} \Phi^{2} (u_{n}')_+ \ dV_{\mathbb{H}^{N}}+o(1).
 \end{equation}

Now using (\ref{eq:5.4}), H\"{o}lder inequality and Poincar\`{e}-Sobolev inequality, we get
\[\begin{array}{ll}
&\ \  \|\Phi (u_{n}')_+\|_{H^{1}(\mathbb{H}^{N})}^{2}\\[2mm]
&\leq(\int_{\mathbb{H}^{N}}[\Phi (v_{n}')_+]^{p+1}\ dV_{\mathbb{H}^{N}})^{p+1}(\int_{\mathbb{H}^{N}}[\Phi (u_{n}')_+]^{p+1}\ dV_{\mathbb{H}^{N}})^{p+1}(\int_{A(a,\sqrt{3})}[(v_{n}')_+]^{p+1}\ dV_{\mathbb{H}^{N}})^{\frac{p-1}{p+1}}\\[2mm]
&\leq S_{N,p}^{-2}\|\Phi (v_{n}')_+\|_{H^{1}(\mathbb{H}^{N})}\|\Phi (u_{n}')_+\|_{H^{1}(\mathbb{H}^{N})}(\int_{A(a,\sqrt{3})}[(v_{n}')_+]^{p+1}\ dV_{\mathbb{H}^{N}})^{\frac{p-1}{p+1}}.
\end{array}\]
Similarly,
\[\begin{array}{ll}
&\ \  \|\Phi (v_{n}')_+\|_{H^{1}(\mathbb{H}^{N})}^{2}\\[2mm]
&\leq S_{N,p}^{-2}\|\Phi (v_{n}')_+\|_{H^{1}(\mathbb{H}^{N})}\|\Phi (u_{n}')_+\|_{H^{1}(\mathbb{H}^{N})}(\int_{A(a,\sqrt{3})}[(u_{n}')_+]^{q+1}\ dV_{\mathbb{H}^{N}})^{\frac{q-1}{q+1}}.
\end{array}\]

Now, if $\|\Phi (v_{n}')_+\|_{H^{1}(\mathbb{H}^{N})} \not\rightarrow 0$ and $\|\Phi (u_{n}')_+\|_{H^{1}(\mathbb{H}^{N})}\not\rightarrow 0$ as $n\longrightarrow \infty$. We get
\[\begin{array}{ll} S_{N,p}^{2}&\leq(\int_{A(a,\sqrt{3})}[(v_{n}')_+]^{p+1}\ dV_{\mathbb{H}^{N}})^{\frac{p-1}{p+1}}(\int_{A(a,\sqrt{3})}[(u_{n}')_+]^{q+1}\ dV_{\mathbb{H}^{N}})^{\frac{q-1}{q+1}}\\[2mm]
&<\delta_{1}^{\frac{p-1}{p+1}}\delta_{2}^{\frac{q-1}{q+1}}\\[2mm]
&<S_{N,p}^{2}.
\end{array}\]
Which is a contradiction. This implies that $\int_{\mathbb{H}^{N}}[\Phi (v_{n}')_+]^{p+1}\ dV_{\mathbb{H}^{N}}\rightarrow 0$  and $\int_{\mathbb{H}^{N}}[\Phi (u_{n}')_+]^{q+1}\ dV_{\mathbb{H}^{N}}\rightarrow 0$. Since $a\in S_{\sqrt{3}}$ is arbitrary, the claim follows.

If $1<p,q<\frac{N+2}{N-2}$, this together with the fact that
\[z_{n}'=(u_{n}',v_{n}')\longrightarrow(0,0),\,\,\,in \,\,L_{loc}^{q}\times L_{loc}^{q}(\mathbb{H}^{N}),\,\,\,2\leq q<\frac{N+2}{N-2},\]
immediately gives a contradiction to (\ref{eq:5.3}). Which implies that $z'=(u',v')\neq(0,0)$.

Thus, we have
\[\begin{array}{ll}
I^{\infty}
&=\lim\limits_{n\longrightarrow \infty}I(z_{n}')\\[2mm]
&\geq\liminf\limits_{n\longrightarrow \infty}(\frac{1}{2}-\frac{1}{p+1})\int_{\mathbb{H}^{N}}[(v_{n}')_+]^{p+1}\ dV_{\mathbb{H}^{N}}+(\frac{1}{2}-\frac{1}{q+1})\int_{\mathbb{H}^{N}}[(u_{n}')_+]^{q+1}\ dV_{\mathbb{H}^{N}}\\[2mm]
&\geq I(z').
\end{array}\]
Consequently $I(z')$ is obtained.

To show that $I^{\infty}>0$, we notice only that if $z'=(u',v')\neq(0,0)$ with $I(z')\neq 0$, then
\[I(z')=(\frac{1}{2}-\frac{1}{p+1})\int_{\mathbb{H}^{N}}[(v')_+]^{p+1}\ dV_{\mathbb{H}^{N}}+(\frac{1}{2}-\frac{1}{q+1})\int_{\mathbb{H}^{N}}[(u')_+]^{q+1}\ dV_{\mathbb{H}^{N}}>0.\]

\textbf{Proof of Theorem \ref{tm:1.4}} By Lemma  \ref{lm:5.1}, we know that there exists a nontrivial solution $z=(u,v)$  of problem (\ref{eq:1.1}) with $I(z)=I^\infty$. Moreover, we have $z=(u, v)>(0,0)$. Otherwise, we may assume that $u$ changes sign, $u=u_+-u_-, u_+\not\equiv 0, u_-\not\equiv 0$ satisfying
\[-\int_{\mathbb{H}^{N}}|\nabla_{\mathbb{H}^N} u_-|\ dV_{\mathbb{H}^N}=\int_{\mathbb{H}^{N}} (v_+)^p u_-\ dV_{\mathbb{H}^N}\geq 0. \]
So we have that $u_-=0$, that is $u\geq 0$. Similarly, we have $v\geq 0$.

\end{proof}

\end{document}